\documentclass[11pt]{amsart}

\usepackage{geometry}
\usepackage{amssymb}
\usepackage{mathtools}
\usepackage{amsthm}
\usepackage{amsmath}
\usepackage{amsfonts}
\usepackage{amscd}
\usepackage{hyperref}
\usepackage{color}
\usepackage{mathtools,caption, commath}
\usepackage[table,dvipsnames]{xcolor}
\hypersetup{
 colorlinks=true,
 linkcolor=DarkOrchid,
 filecolor=blue,
 citecolor=olive,
 urlcolor=orange,
 pdftitle={RNT Project},
 }

\DeclareMathOperator{\Gal}{Gal} 
\DeclareMathOperator{\N}{\mathsf{N}} 
\newcommand{\Z}{\mathbb{Z}}
\newcommand{\cyc}{\operatorname{cyc}}
\newcommand{\ac}{\operatorname{ac}}
\newcommand{\Aut}{\operatorname{Aut}}

\newtheorem{theorem}{Theorem}[section]
\newtheorem{proposition}[theorem]{Proposition}

\newtheorem{lemma}[theorem]{Lemma}

\theoremstyle{remark}
\newtheorem{remark}[theorem]{Remark}
\newtheorem{Example}[theorem]{Example}

\begin{document}
\begin{abstract} 
Let $K$ be an imaginary quadratic field in which the odd prime $p$ does not split.
When the $p$-part of the class group of $K$ is cyclic, we describe the possible structures
for the $p$-part of the class group of the first level of the cyclotomic $\mathbb Z_p$-extension of $K$.
This allows us to show the compatibility of the heuristics of Cohen--Lenstra--Martinet for class groups with the heuristics of Ellenberg--Jain--Venkatesh for how often the cyclotomic Iwasawa invariant $\lambda$ equals 1.
\end{abstract}

\title{The first level of $\mathbb Z_p$-extensions and compatibility of heuristics}

\author[Kundu]{Debanjana Kundu}
\address{University of Texas Rio Grande Valley, Edinburg, TX, 78539, USA}\email{dkundu@math.toronto.edu}
\author[Washington]{Lawrence C. Washington}
\address{Dept. of Mathematics, Univ. of Maryland, College Park, MD, 20742 USA}\email{lcw@umd.edu}

\maketitle

\section{Introduction}

\subsection{Background}
Let $K=K_0$ be an imaginary quadratic field and fix an odd prime $p$ that does not split in $K/\mathbb Q$.
Let $K_{\cyc}/K$ be the cyclotomic $\mathbb Z_p$-extension of $K$, write $A_n$ to denote the Sylow $p$-subgroup of the class group of $K_n$, and let $p^{e_n}$ be the order of $A_n$.
A classical theorem of K.~Iwasawa, supplemented by a result of B.~Ferrero and the second author (see \cite{FW79}), says that there exist integers $\lambda\ge 0$ and $\nu$, independent of $n$, such that for all sufficiently large $n$, 
\[
e_n=\lambda n+\nu.
\]
There are two natural questions to ask: 
\begin{enumerate}
    \item[(I)] What is the behavior of the $p$-part of the class group in the initial layers?
    \item[(II)] What is the distribution of the $\lambda$-invariant as $K$ varies?
\end{enumerate}
Partial results in the direction of (I) appear in \cite{G, San91}. More recently, J.~Ellenberg--S.~Jain--A.~Venkatesh made progress towards (II).
In \cite{EJV}, they used a random matrix model to predict the probability that $\lambda$ takes on a given value as $K$ varies.
For a precise statement of their heuristics, see Section~\ref{alternative}.

To get a complete answer for (I) it would be useful to get a holistic picture of the possible group structures of $A_1, A_2, \ldots$ for given $A_0$. 
Once the base field is fixed (alternatively once the class group of $K$ is fixed to have a certain structure), one expects that there are some restrictions on the possible group structures for the class group of the first layer of the cyclotomic $\Z_p$-extension.
This is the first question we address in this article.

\smallskip 

\textbf{Question 1a.} Given $A_0$, what are the possible structures of $A_1, A_2, \ldots$?

\smallskip

Answering this question in full generality is a difficult one,
so we restrict ourselves to the situation that there is a unique prime in $K$ above $p$ that is totally ramified in $K_{\cyc}/K$.
In this setting, Iwasawa proved that $\lambda=0$ if and only if $A_0=0$; see \cite{Iwa56}.
Then again, work of R.~Gold \cite{G} says that $\lambda=1$ if and only if $A_0$ and $A_1$ are cyclic of order greater than 1 (see Proposition~\ref{7.1}).
These previous results indicate that in our situation of interest, we can obtain information on the Iwasawa invariants from the base field and the first layer of the cyclotomic tower. 
The group $A_1$ is not only an abelian group, but also a module over the $p$-adic group ring $\mathbb Z_p[\Gal(K_1/K_0)]$.

In this paper, we restrict ourselves to the case that $A_0$ is cyclic.
We then completely determine the possible structures of $A_1$, which leads to the following natural question:

\smallskip 

\textbf{Question 1b.} How often do the different possibilities of $A_1$ arise?

\smallskip

H.~Cohen and H.~Lenstra \cite{CL} put forth heuristics that predict that the probability $A_0$ is a given abelian $p$-group is inversely proportional to the size of the automorphism group of this group; see also Section~\ref{alternative}.
These heuristics were later modified in \cite{CM} by H.~Cohen--J.~Martinet.
The philosophy of Cohen--Lenstra--Martinet is that the frequencies of various possibilities for $A_1$ should be inversely proportional to the size of their automorphism groups as modules over this group ring; see also \cite{Gr}.
This raises the following question which we also investigate using our answer(s) to the previous two questions

\smallskip

\textbf{Question 2.}
Are the Ellenberg--Jain--Venkatesh heuristics compatible with Cohen--Lenstra--Martinet heuristics for $A_0$ and $A_1$?

\smallskip

In Section~\ref{compatibility}, we see this is the case for $\lambda=1$.

\subsection{Our results and organization}
In this article, we consider the case where $A_0$ is cyclic\footnote{We will say cyclic to mean cyclic \emph{and} non-trivial}.
In Sections~\ref{basic}, \ref{prelims},  and \ref{reiner}, we give, in a more general setting, the list of the possibilities for $A_1$, given that $A_0$ is cyclic of order $p^m$.
The answers are first given in terms of I.~Reiner's classification of ideals of the $p$-adic group ring $\mathbb Z_p[\Gal(K_1/K_0)]$; see \cite{R}. 
In Section~\ref{groups}, we translate the results given by Reiner's classification into the standard form for abelian $p$-groups.
In Section~\ref{compatibility}, we return to $\mathbb Z_p$-extensions and show that the Cohen--Lenstra--Martinet heuristics yield the prediction
of Ellenberg--Jain--Venkatesh for how often $\lambda=1$.
In Section~\ref{alternative}, we present an alternative heuristic for how often $\lambda$ takes on a given value.
This more na{\"i}ve heuristic agrees with the Ellenberg--Jain--Venkatesh heuristic for $\lambda=0$ and $\lambda=1$, and differs slightly for larger values of $\lambda$; this difference is so slight that it is difficult to differentiate the two predictions via computations.
Extending our results to study $\lambda=2$ via Cohen--Lenstra--Martinet heuristics would distinguish between the two possible heuristics, but this might be a much more technical task.
In Section~\ref{data}, we give data on the occurrence of various possibilities for $A_1$ in order to support the use of the Cohen--Lenstra--Martinet philosophy.
In Section~\ref{comparison sec} we compare the cyclotomic situation with the anti-cyclotomic one.
Finally, in Section~\ref{ff} we mention briefly the situation in function fields.

\section*{Acknowledgements}
We thank Brin Mathematics Research Center for hosting the conference Vistas in Number Theory in June 2024 and the Fields Institute for hosting the Canadian Number Theory Association Conference in June 2024, and thereby providing productive working environments -- much progress towards this article was made during those two weeks.
We thank Jeff Achter, Bryden Cais, and Rene Schoof for helpful discussions.

\section{The main theorem}
\label{basic}

Although our main application is to cyclotomic $\mathbb Z_p$-extensions, we state our main theorem in a more general setting.
For example, we do not assume that the unique ramifying prime lies above $p$, even though this is what happens for $\mathbb Z_p$-extensions.
\begin{theorem}
\label{main theorem}
Let $p$ be an odd prime and let $K$ be an imaginary quadratic field.
Suppose that the following hypotheses hold:
\begin{enumerate}
\item[(A)] $K_1/K$ is a Galois extension of degree $p$ in which only one prime ramifies and this prime is totally ramified.
\item[(B)] the natural map $A_0\to A_1$ is injective, where $A_0$ (resp. $A_1$) denotes the $p$-part of the ideal class group  of $K$ (resp. $K_1$).
\item[(C)] $A_0$ is cyclic of order $p^m$ with $m\ge 1$.
\end{enumerate}
Then $A_1$ is one of the following groups:
\begin{gather*}
\left(\mathbb Z/p^{m}\mathbb Z\right) ^p,
\end{gather*}
\begin{gather*}
\left(\mathbb Z/p^{m-1}\mathbb Z\right) \times \left(\mathbb Z/p^{s+1}\mathbb Z\right)^{a}\times 
\left(\mathbb Z/p^s\mathbb Z\right)^{p-1-a} \\
\text{ with }  m\le s \text{ and } 1\le a\le p-1,
\end{gather*}
\begin{gather*}
\left(\mathbb Z/p^{m+1}\mathbb Z\right) \times \left(\mathbb Z/p^{s+1}\mathbb Z\right)^{b}\times 
\left(\mathbb Z/p^s\mathbb Z\right)^{p-1-b} \\
 \text{ with } 0\le s < m \text{ and } 0\le b\le p-2, \text{ and with } b\ne p-2 \text{ if } m=s+1.
\end{gather*}

\end{theorem}

The case $s=0, \ b=0$ gives the possibility that $A_1$ is cyclic of order $p^{m+1}$. This is the case that corresponds to $\lambda = 1$ when we discuss $\mathbb Z_p$-extensions in Section \ref{compatibility}. The proof of the theorem, or a straightforward calculation, shows that when $A_0$ is cyclic of a fixed order $p^m$, then the structure of $A_1$ is uniquely determined by $\abs{A_1}$, except when $\abs{A_1}=p^{mp}$.
In this case, there are two possibilities: $\left(\mathbb Z/p^{m}\mathbb Z\right) ^p$ and $\left(\mathbb Z/p^{m-1}\mathbb Z\right) \times \left(\mathbb Z/p^{m+1}\mathbb Z\right) \times  \left(\mathbb Z/p^m\mathbb Z\right)^{p-2}$.

\begin{remark}
The restriction that $b\ne p-2$ when $m=s+1$ is not necessary since the resulting group is included in the case $m=s$, $a=1$.
However, we include this condition to remove duplication from the list.
\end{remark}

The proof of the theorem occupies the next few sections.

\section{Preliminaries on class groups}
\label{prelims}

With notation introduced so far, let $\sigma$ generate $\Gal(K_1/K)$, so $\sigma^p=1$.
We will often use the norm element
\[
\N=1+\sigma+\sigma^2+\cdots +\sigma^{p-1}.
\]
$A_1$ is a module for the local ring $\mathbb Z_p[\sigma]$.
We identify $A_0$ with its image in $A_1$. 
This is possible because we have assumed that $A_0$ injects into $A_1$.
Note that $A_0$ is fixed by $\sigma$.
In the proof of the next result, we crucially use several hypotheses made in Theorem~\ref{main theorem}.

\begin{proposition}
\label{prop1}
Let the notations be as in Theorem~\ref{main theorem}.
The following assertions hold.
\begin{enumerate}
\item[(a)] The norm map $\N: A_1\to A_0$ is surjective.
\item[(b)] $A_0=\{a\in A_1\, | \, \sigma a = a\} =  \operatorname{Image}(\N)$.
\item[(c)] There is an ideal $I$ of $\mathbb Z_p[\sigma]$ such that $A_1\simeq \mathbb Z_p[\sigma]/I$.
\end{enumerate}
\end{proposition}

\begin{proof} 
(a) Since $K_1/K$ is assumed to be totally ramified, the norm map on ideal classes is surjective; see \cite[Theorem~10.1]{W}.

\smallskip

(b) Since only one prime ramifies in $K_1/K$ and since the units of $K$ are $\pm 1$ (except for the trivial cases $\mathbb Q(\sqrt{-1})$ and $\mathbb Q(\sqrt{-3})$), Chevalley's formula (see for example \cite{lemmermeyer2013ambiguous}) says the number of classes fixed by the Galois group $\langle \sigma \rangle$ is
\[
\abs{A_1^{\langle \sigma\rangle}} = \frac{\abs{A_0} p}{p \, [E_{K} : E_{K}\cap \text{Norm}(K_1^{\times})]} = \abs{A_0}.
\]
Here, $E_K$ denotes the unit group of $K$. 
Since $A_0\subseteq A_1^{\langle\sigma\rangle}$, the two groups are equal.

\smallskip

(c) $A_1$ is a finitely generated module over the local ring $\mathbb Z_p[\sigma]$. The maximal ideal of $\mathbb Z_p[\sigma]$ is generated by $p$ and $\sigma-1$.
Part (b) implies that 
\[
\abs{(\sigma - 1)A_1} = \frac{\abs{A_1}}{\abs{A_1^{\langle\sigma\rangle}} }= \frac{\abs{A_1}}{\abs{A_0}}.
\]
On the other hand, part (a) implies that
\[
\frac{\abs{A_1}}{\abs{\ker(\N)}} = \abs{A_0}.
\]
Since $(\sigma - 1)A_1\subseteq \ker(\N)$ and they have the same order, they are equal.
Therefore, the norm yields an isomorphism
\[
A_1/(\sigma - 1)A_1 \simeq A_0 \simeq \mathbb Z/p^m\mathbb Z,
\]
so  $A_1/(p, \sigma - 1)A_1$ is cyclic.
Nakayama's Lemma implies $A_1$ is cyclic as $\mathbb Z_p[\sigma]$-module.
\qedhere
\end{proof}

Next, to find the possibilities for $A_1$, we need to find the ideals $I$ of $\mathbb Z_p[\sigma]$ satisfying the following two conditions:
\begin{enumerate}
\item[(i)] $\N(\mathbb Z_p[\sigma]/I)$ has order $p^m$
\item[(ii)] $\N(\mathbb Z_p[\sigma]/I) = (\mathbb Z_p[\sigma]/I)^{\langle \sigma\rangle}$.
\end{enumerate}
Since $\sigma x \equiv x\mod I$ is equivalent to $(\sigma - 1)x\in I$, the second condition can be restated as
\begin{enumerate}
\item[(ii')] If $x\in\mathbb Z_p[\sigma]$ and $(\sigma-1)x\in I$, then there exists $k\in \mathbb Z_p$ such that $x-k\N\in I$.
\end{enumerate}

\section{The arithmetic of \texorpdfstring{$\mathbb Z_p[\pi]$}{}}
\label{Zzeta}

Let $\zeta$ be a primitive $p$-th root of unity and let $\pi=\zeta-1$. We have
\[
1=\zeta^p=(1+\pi)^p \equiv 1+p\pi + \pi^p \pmod {p\pi^2}.
\]
Therefore, 
\[
\pi^{p-1} = up,
\]
with $u\in \mathbb Z_p[\zeta]^{\times}$ satisfying $u\equiv -1\pmod {\pi}$. 
For $s\ge 0$, we can write
\[
u^s= a_0+a_1\pi + a_2\pi^2+ \cdots + a_{p-2} \pi^{p-2}
\]
with $a_i\in \mathbb Z_p$ for each $i$ and $p\nmid a_0$.
Let  $r$ be a positive integer; write 
\[
r=(p-1)s+t \text{  with } s\ge 0 \text{ and } 0\le t < p-1.
\]
Then 
\begin{align*}
\pi^r&=u^sp^s\pi^t \\
& = p^s \pi^t(a_0+a_1\pi + a_2\pi^2+ \cdots + a_{p-2} \pi^{p-2}) \\
&= a_0 p^s\pi^t+\cdots + p^sa_{p-t-2}\pi^{p-2} + up^{s+1}a_{p-t-1}+ \cdots + up^{s+1}a_{p-2}\pi^{t-1}.
\end{align*}

Observe that we gain an extra factor of $p$ after the powers of $\pi$ wrap around and are expressed in terms of $p$ times a low power of $\pi$.
Therefore, each term starting with $up^{s+1}a_{p-t+1}$ has a factor of $p^{s+1}$. Expressing each $u\pi^i$   for $0\le i \le t-1$
as a polynomial in $\pi$, we obtain
\[
\pi^r = b_0 p^{s+1} + \cdots + b_{t-1} p^{s+1} \pi^{t-1} + b_{t}p^s \pi^{t} + \cdots + b_{p-2} p^s \pi^{p-2}
\]
with $b_i\in \mathbb Z_p$ and $p\nmid b_t$.
We can write this as
\[
\pi^r = O(p^{s+1}) + O(p^{s+1})\pi + \cdots + u_t p^s \pi^t + O(p^s)\pi^{t+1} + \cdots + O(p^s) \pi^{p-2},
\]
where $u_t$ is a $p$-adic unit and where $O(p^s)$ represents an unspecified multiple of $p^s$ in $\mathbb Z_p$.

\subsection*{Change of basis matrix:}
As motivation for later calculations with the group ring $\mathbb Z_p[\sigma]$, we make the following computation.
Of course, 
\[
B=\{1, \pi, \pi^2, \dots, \pi^{p-2}\}
\]
is a $\mathbb Z_p$-basis for $\mathbb Z_p[\pi]$.
Now consider $\pi^r\mathbb Z_p[\pi]$.
Writing the $\mathbb Z_p$-basis $\{\pi^r, \pi^{r+1}, \dots, \pi^{r+p-2}\}$ of $\pi^r\mathbb Z_p[\pi]$ in terms of the basis $B$ 
of $\mathbb Z_p[\pi]$ yields the $(p-1)\times (p-1)$ matrix
\[
\begin{pmatrix}
 O(p^{s+1}) & \cdots & O(p^{s+1})  & u_0 p^{s+1} & \cdots & O(p^{s+2}) \\
\cdots & \cdots & \cdots & \cdots  & \ddots & \cdots \\
\cdots & \cdots & \cdots & \cdots  & \cdots & u_{t-1}p^{s+1} \\
u_t p^s & \cdots & O(p^{s+1}) & O(p^{s+1}) & \cdots & O(p^{s+1}) \\
\cdots & \ddots & \cdots & \cdots & \cdots \\
O(p^s) & \cdots &  u_{p-2} p^s  & O(p^{s+1}) &  \cdots & O(p^{s+1})\\
\end{pmatrix}
\]
where each $u_i$ is a $p$-adic unit.
For example, the first column represents the above expansion of $\pi^r$ as a linear combination of the powers of $\pi$ in the basis $B$, and the last column represents the expansion of $\pi^{r-1}$.

\subsection*{Obtaining the group structure:}
Subtract a multiple of the last row from each row above it to remove the $O(p^{s+1})$ terms from the column above the $u_{p-2}p^s$ in the bottom row.
Then subtract multiples of the penultimate row from the rows above it to remove the $O(p^{s+1})$ terms in the column that is to the left of the previous one.
Continue subtracting multiples of each row from those above it, to obtain a matrix of the form
\[
\begin{pmatrix}
0 & \cdots & 0  & w_1 p^{s+1} & \cdots & 0 \\
\cdots & \cdots & \cdots  & \cdots & \ddots & \cdots \\
\cdots & \cdots & \cdots  & \cdots & \cdots & w_2 p^{s+1} \\
w_3 p^s & \cdots &0  & O(p^{s+1}) & \cdots & O(p^{s+1}) \\
\cdots & \ddots & \cdots  & \cdots & \cdots \\
O(p^s) & \cdots & w_4 p^s  & O(p^{s+1}) &  \cdots & O(p^{s+1})
\end{pmatrix}
\]
where each $w_i$ is a $p$-adic unit.
Now we can use row operations to remove the  $O(p^s)$ and $O(p^{s+1})$ entries to obtain a matrix with exactly one nonzero entry in each row and each column.
We can read off the group structure from this matrix.
More precisely, we have proved the following result.

\begin{proposition}
Let $\zeta$ be a primitive $p$-th root of unity and $\pi= \zeta-1$.
Let $r$ be a positive integer with $ r=(p-1)s+t$ where $s\ge 0 \text{ and } 0\le t < p-1$.
Then
\[
\mathbb Z_p[\zeta]/\pi^r\mathbb Z_p[\zeta]\simeq \large(\mathbb Z/p^{s+1}\mathbb Z)^t \times \left(\mathbb Z/p^s\mathbb Z\right)^{p-1-t}.
\]    
\end{proposition}

\section{Reiner's classification of ideals of \texorpdfstring{$\mathbb Z_p[\sigma]$}{}}
\label{reiner}

We now describe Reiner's classification \cite{R} of ideals of $\mathbb Z_p[\sigma]$.
Consider the diagram
\[
\begin{CD}
\mathbb Z_p[\sigma] @>\epsilon >>  \mathbb Z_p \\
@VV\phi V@VV\!\!\mod p V  \\
\mathbb Z_p[\zeta] @>>\mod \pi > \mathbb F_p
\end{CD}
\]
where $\epsilon: \mathbb Z_p[\sigma]\to \mathbb Z_p$ is the augmentation map $\sum a_i \sigma^i \mapsto \sum a_i$ and $\phi(\sigma)=\zeta$.
The minimal polynomial of $\zeta$ is $x^{p-1}+x^{p-2}+\cdots + x+1$; hence, the kernel of $\phi$ is $\mathbb Z_p \N$, i.e., the multiples of the norm.
The kernel of $\epsilon$ is generated by $\sigma-1$.
It is easy to see that
\[
\mathbb Z_p[\sigma] \simeq \left\{(x,y)\in \mathbb Z_p[\zeta] \times \mathbb Z_p\, \mid \, \phi(x)\text{ mod }\pi =\epsilon(y)\text{ mod } p \text{ in } \mathbb F_p\right\}.
\]
Therefore, $\mathbb Z_p[\sigma]$ is the fiber product of  $\mathbb Z_p[\zeta]$ and $\mathbb Z_p$ over $\mathbb F_p$.
We often write elements of the group ring as ordered pairs
$(x, y)$ satisfying the compatibility condition.

The action of $\mathbb Z_p[\sigma]$ is given by $\sigma(x,y)=(\zeta x, y)$. Therefore, 
\begin{align*}
(\sigma - 1)(x,y) &= (\pi x, 0) \\
\N(x, y) &= (0, py).
\end{align*}

Let $I$ be an ideal of finite index in $\mathbb Z_p[\sigma]$.
Then $\phi(I)$ is an ideal of the discrete valuation ring $\mathbb Z_p[\zeta]$, hence is of the form $\pi^r\mathbb Z_p[\zeta]$ for some integer $r\ge 0$.
Choose $\alpha\in I$ with $\phi(\alpha)=\pi^r$. Let $b=\epsilon(\alpha)\in \mathbb Z_p$.
Then $\alpha$ is represented in the fiber product by $(\pi^r, b)$.
The choice of $\alpha\in I$ is not unique.
If $\phi(\alpha_1)=\pi^r$, then $\alpha_1-\alpha = k\N$ for some $k\in \mathbb Z_p$ such that $k\N\in I$.
Let $p^{m+1}\mathbb Z_p= \epsilon(\mathbb Z_p \N\cap I)$.
In other words, $p^{m}\N$ is the smallest non-zero multiple of $\N$ that lies in $I$, which exists since $I$ has finite index in the group ring.
Since $\epsilon(\alpha) - \epsilon(\alpha_1)\in p^{m+1}\mathbb Z_p$,
the number $b$ is well-defined mod $p^{m+1}$.
It is easy to see that
\[
I = \mathbb Z_p[\sigma] (\pi^r, b) + \mathbb Z_p (0, p^{m+1})
\]
under the identification of $\mathbb Z_p[\sigma]$ given by the fiber product.  
Note that 
\[
\N(b,b)=(0, pb) = \N(\pi^r, b) \in I,
\]
so $p^m \mid  b$ since $p^m$ is the smallest positive integer with $p^m \N\in I$.
This says we may assume 
\[
b\in \{p^{m}j\, |\, 0\le j\le p-1\}.
\]
If $r=0$, then $(1,b)\in I$.
By the fiber product condition, $b\equiv 1\pmod p$, so $(1,b)$ is a unit and $I=\mathbb Z_p[\sigma]$. 

Since the norm $\N\in \mathbb Z_p[\sigma]$ acts as 0 on $\mathbb Z_p[\zeta]$ and as $p$ on $\mathbb Z_p$,  the image of the norm is 
\[
\N\left(\mathbb Z_p[\sigma]/I\right) = (0, p\mathbb Z_p)/(0, p^{m+1}\mathbb Z_p) \simeq \mathbb Z/p^{m}\mathbb Z.
\]
Since  $\N (A_1)=A_0\ne 0$, we cannot have $m=0$ for ideals $I$ as in Proposition~\ref{prop1}(c).
Henceforth, we ignore the cases where $r=0$ or $m=0$.

In summary, Reiner's classification says the following: 
\begin{proposition}
\label{Reiner}
Let $I$ be an ideal of finite index in $\mathbb Z_p[\sigma]$ such that \[
\N(\mathbb Z_p[\sigma]/I)\simeq \mathbb Z/p^m \mathbb Z \ne 0.
\]
Then there exists
\begin{enumerate}
\item[(a)] an integer $r\ge 1$,
\item[(b)] an integer $b\in p^{m}\mathbb Z_p/p^{m+1}\mathbb Z_p$
\end{enumerate}
such that
\[
I = \mathbb Z_p[\sigma] (\pi^r, b) + \mathbb Z_p (0, p^{m+1}).
\]
Moreover,
\[
\abs{ \mathbb Z_p[\sigma]/I} = p^{r+m}.
\]
\end{proposition}
\begin{proof}
All that remains to be proved is the last statement.
Using the fiber product representation of the group ring, we have
\begin{align*}
I &= \mathbb Z_p[\sigma] (\pi^r, b) + \mathbb Z_p (0, p^{m+1})\\
&\subseteq\mathbb Z_p[\sigma] (1, 1) + \mathbb Z_p (0, p)  = \mathbb Z_p[\sigma]\\
&\subset \mathbb Z_p[\sigma] (1, 1) + \mathbb Z_p (0, 1).
\end{align*}
The ideal $I$ is of index $p^{r+m+1}$ in the sum in last line, and the group ring is of index $p$ in this sum, so
the index of $I$ in the group ring is $p^{r+m}$.\end{proof}

To determine which ideals arise from cyclic $A_0$, we need to determine the ideals $I$ such that  the only elements of $\mathbb Z_p[\sigma]/I$ that are fixed by $\sigma$ are the multiples of $\N$.

\begin{proposition}
\label{prop 5.2}
Let $r, m, b$ be as in Proposition \ref{Reiner} and let $I$ be the corresponding ideal.
Then
\[
\left(\mathbb Z_p[\sigma]/I\right)^{\langle \sigma\rangle} = \N\left(\mathbb Z_p[\sigma]/I\right) \Longleftrightarrow b\not\equiv 0\pmod{p^{m+1}}.
\]
\end{proposition}

\begin{proof}
($\Rightarrow$)
Suppose $b\equiv 0\pmod{p^{m+1}}$.
Set $c=1$ if $r=1$ and $c=0$ if $r>1$.
Then  $(\pi^{r-1}, c)$ satisfies the fiber product condition and we check that $(\sigma-1)(\pi^{r-1}, c) = (\pi^r, 0)$ is an element of $I$. 
If  
\[
(\pi^{r-1},c)- \N(u, v) =(\pi^{r-1}, c-pv)\in I
\]
for some $(u, v)\in \mathbb Z_p[\sigma]$, then 
\[
\pi^{r-1} \in \phi(I) = \pi^r\mathbb Z_p[\zeta],\]
a contradiction.
Therefore, $(\pi^{r-1}, c)$ is not equivalent to a multiple of $\N\pmod{I}$, so we have an element fixed by $\sigma \pmod{I}$ that is not congruent mod $I$ to a multiple of $\N$.

\smallskip

($\Leftarrow$)
Now suppose $b\not\equiv 0\pmod{p^{m+1}}$.
Suppose $(\sigma-1)(x,y)\in I$.
We want to show $(x,y)$ is congruent to a multiple of $\N \pmod{I}$.
Write
\[
(\pi x, 0) =(\sigma-1)(x,y) =  (u,v)(\pi^r, b) + k(0, p^{m+1})
\]
with $(u,v)\in \mathbb Z_p[\sigma]$ and $k\in \mathbb Z_p$. This says 
\[
\pi x= u\pi^r,  \qquad 0=vb+kp^{m+1}.
\]
Since $b\not\equiv 0\pmod{p^{m+1}}$, we must have $p\mid v$.
Since $(u, v)$ lives in the fiber product, $\pi\mid u$.
Therefore, $\pi^r\mid x$. Since $(x,y)$ lives in the fiber product, $p\mid y$. 
Choose $\ell\in \mathbb Z_p$ such that $(u/\pi, \ell)\in \mathbb Z_p[\sigma]$, and choose $z\in \mathbb Z_p$ such that $y=zp+\ell b+kp^{m+1}$.
This is possible because $p^m \mid b$ and $m\ge 1$.
Then
\[
(x,y)= \N(z,z)+(u/\pi ,\ell)(\pi^r, b) + k(0, p^{m+1}).
\]
Therefore, $(x,y) \equiv \N(z,z) \mod{I}$, 
so the fixed elements  are congruent to multiples of $\N$.

Conversely, anything congruent mod $I$ to a multiple of $\N$ is fixed 
mod $I$ by $\sigma$ since $(\sigma -1)\N=0$.
Therefore, the fixed elements mod $I$ are exactly the multiples of $\N$ mod $I$, as claimed.
\end{proof}

Recall that we are assuming $A_0$ is cyclic of order $p^m$ with $m\ge 1$.
Proposition~\ref{prop 5.2} says that for $I$ satisfying conditions (i) and (ii) (in the end of Section~\ref{prelims}) we must have $b\not\equiv 0\pmod{p^{m+1}}$. 
Therefore, the ideals that yield possibilities for $A_1$ have the form 
\[
I = \mathbb Z_p[\sigma](\pi^r, jp^m) + \mathbb Z_p(0, p^{m+1}), \text{ with } r\ge 1 \text{ and } 1\le j\le p-1.
\]

It will be convenient to let $T=\sigma-1$.  
Let $\alpha=T^r+jp^{m-1}\N$.
In the fiber product representation, when $r\ge 1$, we have $\phi(\alpha)= \pi^r$  and $\epsilon(\alpha)=jp^m$, so $\alpha$
corresponds to $(\pi^r, jp^m)$.
Similarly,  $p^m \N$ corresponds to $(0,p^{m+1})$. Since $j^{-1}\N\alpha= p^m\N$, the ideal $I$ is generated by $\alpha$ over $\mathbb Z_p[\sigma]$.

\section{Group structures}\label{groups}

We want to imitate the calculation of the structure of $\mathbb Z_p[\pi]/(\pi^r)$ for $\mathbb Z_p[\sigma]$.
In other words, we wish to find the structure of $\mathbb Z_p[\sigma]/I$, where $I$ is an ideal as in Section \ref{Reiner}.
The main content of this section is summarized in Table~\ref{table 1} at the end of the section.

We continue to let $T=\sigma-1$.  
Then
\[
1=\sigma^p = (T+1)^p = T^p+pT^{p-1}+\cdots +pT+1,
\]
so 
\[
T^p=\tilde{u}pT
\]
with $\tilde{u}\in \mathbb Z_p[T]$ satisfying $\tilde{u}(0)=-1$. 

Every element  $\beta\in \mathbb Z_p[\sigma]$ can be written as a $\mathbb Z_p$-linear combination
\[
\beta=b_0+b_1T+\cdots +b_{p-1}T^{p-1}.
\]
Note that $b_0\ne 0\iff \beta \N\ne 0$; in particular $b_0=0$ whenever $\beta$ is a multiple of $T$.
Thus, there are no constant terms for many elements we consider.

\subsection*{Change of basis matrix}
As shown at the end of Section \ref{Reiner}, $\alpha=T^r+jp^{m-1}\N$ generates $I$ over $\mathbb Z_p[\sigma]$.
We need to find $\mathbb Z_p$-generators for $I$. As before, write
\[
r=(p-1)s+t \text{ with } 0\le t < p-1.
\]
Since $T\N=0$, we have
\[
T\alpha=T^{r+1}=\tilde{u}^sp^sT^{t+1}, \dots, T^{p-1-t}\alpha = \tilde{u}^s p^s T^{p-1}.
\]
We continue to multiply by powers of $T$, but now the powers wrap around back to the beginning to yield
\[
T^{p-t}\alpha =\tilde{u}^{s+1} p^{s+1} T, \dots, T^{p-1} \alpha = \tilde{u}^{s+1} p^{s+1} T^{t}.
\]
We conclude that $I$ is generated over $\mathbb Z_p$ by
\[
\begin{cases}
\tilde{u}^{s+1} p^{s+1} T, \dots, \tilde{u}^{s+1} p^{s+1} T^t, \tilde{u}^{s} p^{s} T^{t+1}, \dots, \tilde{u}^s p^s T^{p-1},  \tilde{u}^sp^sT^t+jp^{m-1}\N \\
 \textrm{ if } t\neq 0\\
 \\
\tilde{u}^{s} p^{s} T, \dots, \tilde{u}^s p^s T^{p-1},   \tilde{u}^{s-1}p^{s-1}T^{p-1}+jp^{m-1}\ N \\
\textrm{ if } t= 0.
\end{cases}
\]
Since 
\[
\tilde{u}^{s+1}\equiv (-1)^{s+1} \pmod{T},
\]
we can write $\tilde{u}^{s+1}=(-1)^{s+1}+c_1T+c_2T^2+\cdots + c_{p-1}T^{p-1}$ with $c_i\in \mathbb Z_p$. 
When $0\le \ell\le t$, we have
\begin{align*}
\tilde{u}^{s+1}T^{\ell} = (-1)^{s+1}T^{\ell}+c_1T^{\ell+1}+\cdots + c_{p-1-\ell}T^{p-1}+c_{p-\ell}T^p+\cdots + c_{p-1}T^{\ell+p-1}.
\end{align*}
Note that 
\[
c_{p-\ell}T^p + \cdots + c_{p-1}T^{\ell+p-1}=c_{p-\ell}\tilde{u}pT + \cdots + c_{p-1}\tilde{u}pT^{\ell},
\]
and each summand $\tilde{u}T^{j}$ with $j>0$ can be written as a $\mathbb Z_p$-linear combination of $T, T^2, \dots, T^{p-1}$.
So
\[
\tilde{u}^{s+1}T^{\ell} =(-1)^{s+1}T^{\ell}+c_1T^{\ell+1}+\cdots + c_{p-1-\ell}T^{p-1}+p\, h(T),
\]
where $h(T)$ is a $\mathbb Z_p$-linear combination of $T, \dots, T^{p-1}$.
Therefore,
\begin{align*}
\tilde{u}^{s+1}&p^{s+1}T^{\ell} \\
= &\big((-1)^{s+1} p^{s+1}+O(p^{s+2})\big) T^{\ell} + O(p^{s+1})T^{\ell +1} + \cdots + O(p^{s+1})T^{p-1}\\
 &  + O(p^{s+2})T + \cdots + O(p^{s+2})T^{\ell-1},
\end{align*}
where $O(p^b)$  denotes a $p$-adic integer with valuation at least $b$.   
Similar expansions hold for $\tilde{u}^s p^s T^{\ell}$ for $t\le \ell\le p-1$.
Since 
\begin{align*}
\N &=1 + (T+1) + \cdots + (T+1)^{p-1}=\frac{(T+1)^{p}-1}{T}\\
& = T^{p-1} + O(p)T^{p-2}+ \cdots + O(p)T + p,
\end{align*}
we have
\[
jp^{m-1}\N = jp^{m-1}T^{p-1}+ O(p^m)T^{p-2}+\cdots + jp^m.
\]

To find the structure of $\mathbb Z_p[\sigma]/I$, we perform the usual row and column operations to find the  structure of a free $\mathbb Z_p$-module modulo some relations.
We express the generators of $I$ in terms of the  $\mathbb Z_p$-basis for $\mathbb Z_p[T]$ consisting of powers of $T$, and put the results in a $p\times p$ matrix:
$$
\begin{pmatrix}  0 & \cdots & 0 &  0 & \cdots & 0 & jp^m\\
 w_1p^{s+1} & \cdots & O(p^{s+2})  &  O(p^{s+1}) & \cdots & O(p^{s+1}) & O(p^{\min(s+1, m)}) \\
 \cdots & \ddots &  \cdots &  \cdots & \cdots & \cdots & \cdots \\
 O(p^{s+1}) & \cdots & w_{t}p^{s+1} &  O(p^{s+1}) & \cdots &  O(p^{s+1}) & (-1)^sp^s\! +\! O(p^{\min(s+1, m)}) \\
O(p^{s+1}) & \cdots & O(p^{s+1}) &  w_{t+1}p^{s} & \cdots &  O(p^{s+1}) & O(p^{\min(s, m)}) \\
 \cdots & \cdots &  \cdots &  \cdots & \ddots & \cdots & \cdots \\
 O(p^{s+1}) & \cdots & O(p^{s+1}) & O(p^{s}) & \cdots &  w_{p-1}p^s & jp^{m-1} + O(p^s)
\end{pmatrix},
$$
when $t\ne 0$, and a slightly modified matrix  when $t=0$. Here, each $w_i$ is a $p$-adic unit. 

These columns are obtained as follows.
The $ij$ entry is the coefficient of $T^i$ in the expansion of a generator of $I$ as a polynomial in $T$ of degree at most $p-1$.
In particular, the first column corresponds to 
\[
T^{p-t}\alpha  = \tilde{u}^{s+1}p^{s+1}T = w_1 p^{s+1}T + O(p^{s+1})T^2 + \cdots + O(p^{s+1})T^{p-1},
\]
whereas the last column corresponds to $\tilde{u}^{s}p^{s}T^{t}+jp^{m-1}\N$. 
When $t=0$, the last column arises from $\tilde{u}^{s-1}p^{s-1}T^{p-1} + jp^{m-1}\N$. 

\begin{Example}
Let $p=5$ and $\alpha=T^6+5\N$; i.e. $m=2, r=6, s=1, t=2$.
The relation $(T+1)^5-1=0$ becomes 
\[
T^5=-5T^4-10T^3-10T^2-5T=5T(-1-2T-2T^2-T^3),
\]
so $\tilde{u}=-1-2T-2T^2-T^3$.
We have
\begin{align*}
\N&=1+(T+1)+(T+1)^2+(T+1)^3+(T+1)^4 \\
&= T^4 + 5T^3+10T^2+10T+5.
\end{align*}
The generators of $I$ are
\begin{align*}
T^7&=-75T-125T^2-105T^3-35T^4\\
T^8 &= 175T+275T^2+225T^3+70T^4\\
T^9 &= -350T-525T^2-425T^3-125T^4\\
T^{10} &= 625T + 900 T^2 +725 T^3 + 200T^4\\
T^6+5\N &= 25+75T+95T^2+65T^3+20T^4
\end{align*}
We put the coefficients of these expressions into a matrix.
Since $T^9=5^2\tilde{u}^{2}T$, we use $T^9$ for the first column.
Since $T^{10}=5^2\tilde{u}^{2}T^2$,
we use $T^{10}$ for the second column.
Similarly, we use $T^7=5\tilde{u}T^3$ and $T^8=\tilde{u}T^4$ for the third and fourth columns.
The last column comes from $T^6+5\N$.
Here is the matrix:
\begin{gather*}
\begin{pmatrix} 0&0&0&0&25\\ -350 & 625 & -75 & 175 & 75\\ -525 & 900  & -125 & 275 & 95\\
-425 & 725 & -105 & 225 & 65\\ -125 & 200 & -35 & 70 & 20
\end{pmatrix} \\
=
\begin{pmatrix} 0 & 0 & 0 & 0 & 25 \\
25w_1 & O(5^3) & O(5^2) & O(5^2) & O(5^2) \\
O(5^2) & 25w_2 & O(5^2) & O(5^2) & -5+O(5^2) \\
O(5^2) & O(5^2) & 5w_3 & O(5^2) & O(5) \\
O(5^2) & O(5^2) & O(5) & 5w_4 & 5+O(5)
\end{pmatrix},
\end{gather*}
where $w_1, w_2, w_3, w_4$ are 5-adic units. Of course, the $5+O(5)$ in the lower right corner could be written as $O(5)$.
The Smith normal form of this matrix is $\text{diag}(5^3\cdot 511, 5^2, 5, 5, 5)$, so
\[
\mathbb Z_5[\sigma]/I\simeq \left( \mathbb Z/5^3\mathbb Z\right)\times\left (\mathbb Z/5^2\mathbb Z\right)\times\left( \mathbb Z/5\mathbb Z\right)\times\left( \mathbb Z/5\mathbb Z\right)\times\left (\mathbb Z/5\mathbb Z\right).
\]
\end{Example}

We now return to the general situation.
In a sense, the $(p-1)\times (p-1)$ matrix obtained by deleting the first row and the last column acts like a lower triangular matrix. 
This motivates the matrix operations we use.
Subtract multiples of the second row from the rows below it to remove each $O(p^{s+1})$ in the first column.
Since this adds $O(p^{s+2})$ to $w_2 p^{s+1}$, for example, we still have a unit times $p^{s+1}$ in the third row, second column. Use this element to remove the entries above and below it.
Continue in this way to obtain 
\[
M=\begin{pmatrix}  0 & \cdots & 0 &  0 & \cdots & 0 & jp^m\\
 w_1p^{s+1} & \cdots & 0  &  0 & \cdots & 0 & O(p^{\min(s+1, m)}) \\
 \cdots & \ddots &  \cdots &  \cdots & \cdots & \cdots & \cdots \\
 0 & \cdots & w_{t}'p^{s+1} & 0& \cdots & 0 & O(p^{\min(s, m)}) \\
0 & \cdots & 0 &  w_{t+1}'p^{s} & \cdots &  0 & O(p^{\min(s, m)}) \\
 \cdots & \cdots &  \cdots &  \cdots & \ddots & \cdots & \cdots \\
0 & \cdots &0 & 0 & \cdots &  w_{p-1}'p^s & jp^{m-1}\! +\! O(p^{\min(s, m)})
\end{pmatrix}
\]
with units $w_i'$.
Note, for example, when the $(t+2)$nd row is used to remove elements above $w_{t+1} p^s$, this row is multiplied at least by $p$.
Therefore, the
entry $O(p^{\min(s,m)})$ in this row contributes $O(p^{1+\min(s,m)})$ to the entries above it in the last column and therefore they still can be represented as $O(p^{\min(s+1, m)})$. 

\subsection{\texorpdfstring{Case $m\le s$ and $t\ne 0$}{}}
The top row can be used to remove other entries from the last column, yielding the matrix
\[
\begin{pmatrix}  0 & \cdots & 0 &  0 & \cdots & 0 & jp^m\\
 w_1p^{s+1} & \cdots & 0  &  0 & \cdots & 0 & 0 \\
 \cdots & \ddots &  \cdots &  \cdots & \cdots & \cdots & \cdots \\
 0 & \cdots & w_{t}'p^{s+1} & 0& \cdots & 0 &0 \\
0 & \cdots & 0 &  w_{t+1}'p^{s} & \cdots &  0 & 0 \\
 \cdots & \cdots &  \cdots &  \cdots & \ddots & \cdots & \cdots \\
0 & \cdots &0 & 0 & \cdots &  w_{p-1}'p^s & jp^{m-1} 
\end{pmatrix}.
\]
Subtract $p$ times the last row from the top row to remove $jp^m$ and this puts $-w_{p-1}'p^{s+1}$ at the top of the next to last column.
Finally, subtract a multiple of the last  column from the next to last column to obtain
\[
\begin{pmatrix}  0 & \cdots & 0 &  0 & \cdots & -w_{p-1}''p^{s+1} &0\\
 w_1p^{s+1} & \cdots & 0  &  0 & \cdots & 0 & 0 \\
 \cdots & \ddots &  \cdots &  \cdots & \cdots & \cdots & \cdots \\
 0 & \cdots & w_{t}'p^{s+1} & 0& \cdots & 0 &0 \\
0 & \cdots & 0 &  w_{t+1}'p^{s} & \cdots &  0 & 0 \\
 \cdots & \cdots &  \cdots &  \cdots & \ddots & \cdots & \cdots \\
0 & \cdots &0 & 0 & \cdots &  0 & jp^{m-1} 
\end{pmatrix}.
\]
Since there is one non-zero entry in each row and each column, we can permute the rows to put these entries on the diagonal.
It follows that
\[
\mathbb Z_p[\sigma]/I \simeq  \left(\mathbb Z/p^{m-1}\mathbb Z\right) \times \left(\mathbb Z/p^{s+1}\mathbb Z\right)^{t+1}\times 
\left(\mathbb Z/p^s\mathbb Z\right)^{p-t-2}.
\]

\subsection{\texorpdfstring{Case $m> s$ and $t\ne 0$}{}}
In the reduction to the matrix $M$, we can be more precise and obtain
\[
\left(0, \dots, w_t' p^{s+1}, 0, \dots, (-1)^sp^s + O(p^{s+1})\right)
\]
as one of the middle rows.
The columns of $M$ can be used to remove all the terms in the last column except for the $jp^m$ in the upper right corner and the $(-1)^sp^s$ in the $(t+1)$st row.
We now have
\[
\begin{pmatrix}  0 & \cdots & 0 &  0 & \cdots & 0 & jp^m\\
 w_1p^{s+1} & \cdots & 0  &  0 & \cdots & 0 & 0  \\
 \cdots & \ddots &  \cdots &  \cdots & \cdots & \cdots & \cdots \\
 0 & \cdots & w_{t}'p^{s+1} & 0& \cdots & 0 & w''p^s \\
0 & \cdots & 0 &  w_{t+1}'p^{s} & \cdots &  0 & 0 \\
 \cdots & \cdots &  \cdots &  \cdots & \ddots & \cdots & \cdots \\
0 & \cdots &0 & 0 & \cdots &  w_{p-1'}p^s & 0
\end{pmatrix}
\]
for some unit $w''$.

Subtract a multiple of the last column from the $t$-th column and then subtract a multiple of the $(t+1)$st row from the top row to obtain
\[
\begin{pmatrix}  0 & \cdots &  wp^{m+1} &0 &  \cdots & 0 & 0 \\
 w_1p^{s+1} & \cdots & 0  &  0 & \cdots & 0 & 0  \\
 \cdots & \ddots &  \cdots &  \cdots & \cdots & \cdots & \cdots \\
 0 & \cdots & 0 & 0& \cdots & 0 & w'' p^s \\
0 & \cdots & 0 &  w_{t+1}'p^{s} & \cdots &  0 & 0 \\
 \cdots & \cdots &  \cdots &  \cdots & \ddots & \cdots & \cdots \\
0 & \cdots &0 & 0 & \cdots &  w'_{p-1}p^s & 0
\end{pmatrix}.
\]
We conclude that
\[
\mathbb Z_p[\sigma]/I \simeq  \left(\mathbb Z/p^{m+1}\mathbb Z\right) \times \left(\mathbb Z/p^{s+1}\mathbb Z\right)^{t-1}\times 
\left(\mathbb Z/p^s\mathbb Z\right)^{p-t}.
\]

\subsection{\texorpdfstring{Case $t= 0$}{}}
We start with the matrix
\[
\begin{pmatrix}  0 & \cdots & 0  & \cdots & 0 & jp^m\\
 w_1p^{s} & \cdots & O(p^{s+1})   & \cdots & O(p^{s+1}) & O(p^{\min(s, m)}) \\
 \cdots & \ddots &  \cdots &  \cdots & \cdots & \cdots \\
 O(p^{s}) & \cdots & w_{i}p^{s} &  \cdots &  O(p^{s+1}) & O(p^{\min(s, m)}) \\
O(p^{s}) & \cdots & O(p^{s}) &  \cdots &  O(p^{s+1}) & O(p^{\min(s, m)}) \\
 \cdots & \cdots &  \cdots & \ddots & \cdots & \cdots \\
 O(p^{s}) & \cdots & O(p^{s}) & \cdots &  w_{p-1}p^s & jp^{m-1}\! +\! (-1)^{s-1}p^{s-1}+O(p^s)
\end{pmatrix}.
\]
Here's how to obtain the last column. We have
\[
T^{(p-1)s} + jp^{m-1}\N = \tilde{u}^{s-1}p^{s-1}T^{p-1} + jp^{m-1}\left(T^{p-1}+O(p)T^{p-2}+\cdots +p\right).
\]
As before, 
\begin{align*}
\tilde{u}^{s-1}T^{p-1} &= T^{p-1}\left((-1)^{s-1} + c_1'T+c_2'T^2+\cdots + c_{p-1}' T^{p-1}\right) \\
&= (-1)^{s-1}T^{p-1}+ O(p)T+O(p)T^2+\cdots + O(p)T^{p-1}
\end{align*}
since $T^{p+j}=O(p)T^{j+1}$ when $j\ge 0$.
Therefore,
\begin{align*}
&T^{(p-1)s} + jp^{m-1}\N \\
& \quad = \left((-1)^{s-1}p^{s-1}+jp^{m-1}+O(p^s)\right) T^{p-1} + O(p^s+ p^m)T^{p-2} + \cdots + jp^{m}.
\end{align*}

Row operations, as before, reduce this to
\[
\begin{pmatrix}  0 & \cdots & 0 &  \cdots & 0 & jp^m\\
 w_1p^{s} & \cdots & 0&   \cdots & 0 & O(p^{\min(s, m)}) \\
 \cdots & \ddots & \cdots & \cdots & \cdots & \cdots \\
 0 & \cdots & w_{i}p^{s} & \cdots &  0& O(p^{\min(s, m)}) \\ \cdots & \cdots &  \cdots & \ddots & \cdots & \cdots \\
 0& \cdots & 0 & \cdots &  w_{p-1}p^s & jp^{m-1}\! +\! (-1)^{s-1}p^{s-1}\! +\! O(p^{\min(s,m)})
\end{pmatrix},
\]

\subsubsection{}
If $m<s$, proceed in exactly the same way as  when $t\ne 0$ to obtain
\[
\mathbb Z_p[\sigma]/I \simeq  \left(\mathbb Z/p^{m-1}\mathbb Z\right) \times \left(\mathbb Z/p^{s+1}\mathbb Z\right)\times 
\left(\mathbb Z/p^s\mathbb Z\right)^{p-2},
\]
which is the answer we obtained for $t\ne 0$, but with $t$ set to $0$.

\subsubsection{}
If $m>s$, we know that the lower right entry has valuation $s-1$, so we can proceed in exactly the same way as  when $t\ne 0$ and obtain
\[
\mathbb Z_p[\sigma]/I \simeq  \left(\mathbb Z/p^{m+1}\mathbb Z\right) \times \left(\mathbb Z/p^s\mathbb Z\right)^{p-2}\times \left(\mathbb Z/p^{s-1}\mathbb Z\right).
\]

\subsubsection{}
Finally, suppose $t=0$ and $m=s$.
If $j\not\equiv (-1)^m$, then the lower right corner has $p$-adic valuation $m-1$.
Use all but the last two columns to remove the middle entries in the last column.
The top and bottom rows of the last two columns
have the form
\[
\begin{pmatrix} 0 & jp^m \\ w_{p-1} p^m & w_p p^{m-1}\end{pmatrix}
\]
and are independent of the rest of the matrix.
Subtract $p$ times a multiple of the bottom row from the top row to change the upper right corner to 0, then subtract $p$ times a multiple of the second column from the first column to change the lower left entry to 0, yielding a matrix of the form
\[
\begin{pmatrix} wp^{m+1} & 0\\  0 & w_p p^{m-1}\end{pmatrix}.
\]
We conclude that
\[
\mathbb Z_p[\sigma]/I \simeq  \left(\mathbb Z/p^{m+1}\mathbb Z\right) \times \left(\mathbb Z/p^{m-1}\mathbb Z\right) \times \left(\mathbb Z/p^m\mathbb Z\right)^{p-2}.
\]

If $j\equiv (-1)^m$, then the lower right corner has $p$-adic valuation at least $m$. The top row removes this entry and we obtain
\[
\mathbb Z_p[\sigma]/I \simeq   \left(\mathbb Z/p^m\mathbb Z\right)^{p}.
\]

\subsection{Summary}
We summarize our results in Table~\ref{table 1}.
Note that the cyclic group $\mathbb Z/p^{m+1}\mathbb Z$ corresponds to $s=0$, $t=1$, hence to the ideals $I=(\pi, jp^m)$. 

\begin{center}
\begin{table}[h!]
\begin{tabular}{|c|c|}
\hline
$m < s$ & $\phantom{\bigg |}\left(\mathbb Z/p^{m-1}\mathbb Z\right) \times \left(\mathbb Z/p^{s+1}\mathbb Z\right)^{t+1}\times 
\left(\mathbb Z/p^s\mathbb Z\right)^{p-t-2}$\\
\hline
$m>s$, $t\ne 0$ & $\phantom{\bigg |}\left(\mathbb Z/p^{m+1}\mathbb Z\right) \times \left(\mathbb Z/p^{s+1}\mathbb Z\right)^{t-1}\times 
\left(\mathbb Z/p^s\mathbb Z\right)^{p-t}$ \\
\hline
$m>s$, $t=0$ & $\phantom{\bigg |}\left(\mathbb Z/p^{m+1}\mathbb Z\right) \times 
\left(\mathbb Z/p^s\mathbb Z\right)^{p-2}\times \left(\mathbb Z/p^{s-1}\mathbb Z\right)$ \\
\hline
$m =s$, $t\ne 0$ & $\phantom{\bigg |}\left(\mathbb Z/p^{m-1}\mathbb Z\right) \times \left(\mathbb Z/p^{m+1}\mathbb Z\right)^{t+1}\times 
\left(\mathbb Z/p^m\mathbb Z\right)^{p-t-2}$\\
\hline
$m=s$, $t=0$, $j\not \equiv (-1)^m$ & $\phantom{\bigg |}\left(\mathbb Z/p^{m+1}\mathbb Z\right) \times \left(\mathbb Z/p^{m-1}\mathbb Z\right) \times \left(\mathbb Z/p^m\mathbb Z\right)^{p-2}$\\
\hline
$m=s$, $t=0$, $j\equiv (-1)^m$ & $ \phantom{\bigg |} \left(\mathbb Z/p^m\mathbb Z\right)^{p}$\\
\hline
\end{tabular}
\medskip
\caption{The possible group structures for $A_1$ when $A_0\simeq \mathbb Z/p^m\mathbb Z$.
Recall that $I$ is generated by $T^r+jp^{m-1}\N$ where $r=(p-1)s+t$ and $0\le t\le p-2$.}
\label{table 1}
\end{table}
\end{center}

The first line of Table~\ref{table 1} gives the groups in Theorem~\ref{main theorem} for $m<s$ and $1\le a=t+1\le p-1$.
The fourth line gives the groups with $m=s$ and $2\le a\le p-1$.
The fifth line gives the case $m=s$ and $a=1$.
The second line of the table gives the groups with $m>s$ and $0\le b=t-1\le p-3$.
If we allow $t=p-1$ in the list of groups given in the second line, then we obtain the groups in the third line, but with $s$ in place of $s+1$.
Since $s<m$, this means the groups in the second line with $t=p-1$, $s+1=m$ are not included among the possibilities.
So the second and third lines  include all cases of $m>s$, $0\le b\le p-2$ in the notation of the theorem, except for the case $m=s+1$, $b=p-2$.
The last line of the table gives the first group listed in Theorem~\ref{main theorem}.
This completes the proof of Theorem~\ref{main theorem}. \hfill $\square$

\section{Compatibility of heuristics}\label{compatibility}

Let $p$ be an odd prime and let $K$ be an imaginary quadratic field in which $p$ is non-split.
Let $K_n$ be the $n$-th layer of the cyclotomic $\mathbb Z_p$-extension $K_{\cyc}$ of $K$.
Let $p^{e_n}$ be the exact power of $p$ dividing the class number of $K_n$. There are integers $\lambda\ge 0$ and $\nu$ such that $e_n=\lambda n+\nu$ for all sufficiently large $n$. 

\cite[Prop.~13.26]{W} asserts that the natural map $A_0\to A_1$ is injective when $K_1/K$ is the start of the cyclotomic $\mathbb Z_p$-extension of $K$.
So, the assumptions of Theorem \ref{main theorem} are satisfied. 
Therefore, if $A_0$ is cyclic, we have a list of possibilities for $A_1$.

\begin{proposition}
\label{7.1}
The following statements are equivalent:
\begin{enumerate}
\item[(a)]
$A_0$ and $A_1$ are cyclic\footnote{Recall that we write a group is cyclic to mean that the group is non-trivial and cyclic.
}. 
\item[(b)]
$\frac{\abs{A_1}}{\abs{A_0}}=p$.
\item[(c)] $\lambda=1$.
\end{enumerate}
\end{proposition}

\begin{proof}
(a) $\Longrightarrow$ (b).
Since the composition $0\ne A_0\hookrightarrow A_1 \twoheadrightarrow A_0$ is multiplication by $p$, we must have $|A_1|>|A_0|$, so $A_1$ is cyclic of order $m_1\ge 2$.
Since $\sigma^p=1$ on $A_1$, we must have $\sigma$ act by multiplication by $1+kp^{m_1-1}$ for some $k$, so 
\begin{equation*}
\begin{split}
\N &=1+\sigma+\sigma^2+\cdots + \sigma^{p-1}\\
&\equiv 1+(1+kp^{m_1-1}) + (1+2kp^{m_1-1}) + \cdots + (1+(p-1)kp^{m_1-1}) \\
&\equiv p \pmod{p^{m_1}}.
\end{split}
\end{equation*}
Therefore, the norm acts on $A_1$ as $p$ times a $p$-adic unit.
Since $A_1$ is cyclic, the kernel of the norm has order $p$.
The fact that the norm is a surjection implies part (b).

(b) $\Longrightarrow$ (c). A result of Gold \cite[Cor.~3]{G} implies that $\lambda=1$.

(c) $\Longrightarrow$ (a).
Let $X = X(K_{\cyc})=\varprojlim A_n$ be the Iwasawa module.
Since $X$ has no finite $\Lambda$-submodules  \cite[Prop.~13.28]{W}, if $\lambda=1$ then $X$ is a cyclic
$\mathbb Z_p$-module.
Both $A_1$ and $A_0$ are quotients of $X$ and must be cyclic.
If $A_0$ is trivial, then Nakayama's Lemma implies that $X$ is trivial; see \cite[Prop.~13.22]{W}.
This is impossible since $\lambda=1$.
Therefore, $A_0$ and hence $A_1$ are nontrivial.
\qedhere
\end{proof}

The Cohen--Lenstra--Martinet heuristics predict that 
\begin{equation}
\label{CL-cyclic}
\tag{CLM: cyclic}
\operatorname{Prob}(A_0 \textrm{ is cyclic}) = p^{-1}(1-p^{-1})^{-2} \prod_{j=1}^{\infty} (1-p^{-j}).
\end{equation}
The Ellenberg--Jain--Venkatesh heuristics predict that 
\begin{equation}
\label{EJV-lambda=1}
\tag{EJV: $\lambda=1$}
\operatorname{Prob}(\lambda =1) = p^{-1} \prod_{j=2}^{\infty} (1-p^{-j}).
\end{equation}

In view of Proposition~\ref{7.1} we know that
\[
\lambda=1 \Longleftrightarrow  \text{ there exists } m\ge 1 \text{ such that } A_0\simeq\mathbb Z/p^m\mathbb Z \text{ and } A_1\simeq \mathbb Z/p^{m+1}\mathbb Z.
\]
Therefore, the conditional probability that $\lambda = 1$ given that $A_0$ is cyclic equals the conditional probability that $A_1$ is cyclic given that $A_0$ is cyclic. 
Combining the two heuristics yields
\begin{align*}
\text{Prob}(\lambda = 1 \, | \, A_0 \text{ is cyclic}) &=  \frac{\text{Prob}(\lambda = 1 \textrm{ and } A_0 \text{ is cyclic})}{ \text{Prob}(A_0 \text{ is cyclic)}} \\
& = \frac{\text{Prob}(\lambda = 1)}{ \text{Prob}(A_0 \text{ is cyclic)}}  = \frac{p-1}{p}.
\end{align*}

In what follows we show that if we weight the possibilities for $A_1$ by the inverse of the sizes of their automorphism groups as $\mathbb Z_p[\sigma]$-modules,
we obtain the same fraction.
This shows a compatibility between \eqref{CL-cyclic} heuristics and \eqref{EJV-lambda=1} heuristics.

The automorphism group of $\mathbb Z_p[\sigma]/I$ as a $\mathbb Z_p[\sigma]$-module 
is  $\left(\mathbb Z_p[\sigma]/I\right)^{\times}$. 

\begin{lemma}\label{autsize}
If  $\mathbb Z_p[\sigma]/I$ has order $p^{m+r}$, then the order of its
automorphism group is $(p-1) p^{m+r-1}$. 
\end{lemma}

\begin{proof}
Let $\mathfrak{m}$ be the maximal ideal of $\mathbb Z_p[\sigma]$ generated by $p$ and $\sigma - 1$. 
Define 
\[
E= \mathfrak{m}/I\subset \mathbb Z_p[\sigma]/I.
\]
Each element of $E$, when raised to the $p$-th power, lies in $p\mathbb Z_p[\sigma]/I$ and therefore cannot be invertible mod $I$.
The complement of $\mathfrak{m}$ in the group ring is the set of invertible elements.
Therefore, the complement of $E$ is $\left(\mathbb Z_p[\sigma]/I\right)^{\times}$.
Observe that $\mathbb Z_p[\sigma]/\mathfrak{m}$ has order $p$.
Therefore, $\abs{E}= p^{m+r-1}$, and its complement has order $p^{m+r}-p^{m+r-1}$.
\end{proof}

For each $r$ there are $p-1$ choices for $b=jp^{m-1}$.
Therefore, the sum over the possibilities for $A_1$, given that $A_0$ is cyclic of order $p^m$, is
\begin{equation}
\label{used later}
\sum_{A_0 \text{ cyclic }p^m} \frac{1}{\abs{\Aut(A_1)}} =  \sum_{ r=1} ^{\infty}  \frac{p-1}{(p-1)p^{r+m-1}} = \frac{1}{(p-1)p^{m-1}}.
\end{equation}
There are $p-1$ ideals that yield $A_1$ cyclic of order $p^{m+1}$.
By Proposition~\ref{7.1}(a), these are all the possible cyclic $A_1$, given $A_0 \simeq \mathbb{Z}/p^m\mathbb{Z}$.  
The automorphism groups have order $(p-1)p^{m}$, so these $p-1$ groups have total weight $(p-1)/(p-1)p^m$.
Therefore, we obtain the heuristic prediction 

\begin{align*}
\text{Prob}(A_1 \text{ is cyclic } | & \, A_0\text{ is cyclic of order } p^m) \\
&= \frac{\text{Prob}(A_1 \text{ is cyclic } \cap A_0\text{ is cyclic of order } p^m)}{\text{Prob}(A_0\text{ is cyclic of order } p^m) } \\
&= \frac{(p-1)/(p-1)p^{m}}{1/(p-1)p^{m-1}} \\
&= \frac{p-1}{p},
\end{align*}
which agrees with the previous estimate. 

This is also evidence that the possible structures that we found for $A_1$ all occur and occur with the expected frequencies; compare with Section~\ref{data}.

\section{An alternative heuristic}\label{alternative}

In this section, we present an alternative to the Ellenberg--Jain--Venkatesh (EJV) heuristics for the cyclotomic $\lambda$-invariant.
We suspect that the EJV is the correct heuristic since it is more structural, but it is unclear whether computations  can conclusively show that one is more likely than the other.
If the theorem on the possible structures for $A_1$ could be extended to the case where $A_0$ has rank 2, then it would be possible to distinguish which heuristic prediction for the $\lambda$-invariant is compatible with the Cohen--Lenstra--Martinet heuristic. 

Let $\eta=\displaystyle\prod_{j=1}^{\infty} (1-p^{-j})$.
The general Ellenberg--Jain--Venkatesh prediction is the following:
\begin{equation}
\label{EJV-gen}
\tag{EJV}
\operatorname{Prob}(\lambda=r) = p^{-r}\eta\prod_{j=1}^r (1-p^{-j})^{-1}.
\end{equation}
The Cohen--Lenstra--Martinet heuristics say that
\begin{equation}
\label{CL-gen}
\tag{CLM}
\operatorname{Prob}(\operatorname{rank} = r) = p^{-r^2} \eta \prod_{j=1}^r (1-p^{-j})^{-2}.
\end{equation}
Since $X$ has no finite $\Lambda$-submodules, we know that $\lambda \ge \operatorname{rank}$; see \cite[Remark after Cor.~5.6]{Gre99}.
Here, by `rank' we mean the $p$-rank of the class group of $K$.

Consider the Iwasawa power series $c_0+c_1T+c_2T^2+\cdots$ attached to $K_{\cyc}/K$.
If the $p$-rank of the class group of $K_0$ is $k$, we know that the coefficients $c_0, \ldots, c_{k-1}$ must all be divisible by $p$. If $\lambda=r$, the coefficients $c_k$ to $c_{r-1}$ are divisible by $p$ and $p\nmid c_r$.
Suppose each $c_i$ with $i\ge k$ is divisible by $p$ with probability $1/p$.
Then the probability that $c_k, \dots, c_{r-1}$ are divisible by $p$ and
$c_r$ is not divisible by $p$ is
\[
\left(\frac{1}{p}\right)^{r-k}\frac{p-1}{p}.
\]
Putting this together with all ranks up to $r$, we find the probability that $\lambda=r$ is 
\begin{equation}
\label{new heuristics}
\tag{new}
\operatorname{Prob}(\lambda = r) = 
\sum_{k=1}^r \text{Prob}(\text{rank}=k)\left(\frac{1}{p}\right)^{r-k}\frac{p-1}{p}.
\end{equation}
When $r=1$, the above heuristics give 
\[
\operatorname{Prob}(\lambda = 1) = \frac{p^{-1}\eta (1-p^{-1})^{-2}(p-1)}{p} = p^{-1}\eta (1-p^{-1})^{-1}.
\]
We see that in this case \eqref{new heuristics} agrees with the \eqref{EJV-lambda=1} heuristic.

In the tables below, we compare how \eqref{new heuristics} compares with \eqref{EJV-gen} and with computational data when $p=3$.
In the following table we compute the values of \eqref{EJV-gen} for some small values of $r$:

\medskip

\begin{tabular}{|c|c|c|c|c|c|c|}
\hline
$r$ & 1 & 2 & 3 & 4 & 5 & 6 \\
\hline
Fraction &&&&&&\\ with $\lambda=r$ & 0.28006 & .10502 & 0.03635  & 0.01227 & 0.00411 & 0.00137\\
\hline
\end{tabular}

\bigskip

For the new heuristic, we have the following predictions:

\medskip

\begin{tabular}{|c|c|c|c|c|c|c|}
\hline
$r$ & 1 & 2 & 3 & 4 & 5 & 6 \\
\hline
Fraction &&&&&&\\ with $\lambda=r$ & 0.28006 & 0.10648 & 0.03555  & 0.01185 & 0.00395 & 0.00132\\
\hline
\end{tabular}

\bigskip

The computed data (taken from \cite{EJV}):

\medskip

\begin{tabular}{|c|c|c|c|c|c|c|}
\hline
$r$ & 1 & 2 & 3 & 4 & 5 & 6 \\
\hline
Fraction &&&&&&\\ with $\lambda=r$ & 0.2680\phantom{;} & 0.0936 \phantom{;}& 0.0322\phantom{;.}  &0.0109\phantom{.} &  0.0035\phantom{;.} & 0.0012\phantom{;.}\\
\hline
\end{tabular}

\section{Data}
\label{data}

To show the compatibility of heuristics of Section \ref{compatibility}, we assumed that the frequency of occurrence of a possibility for $A_1$ is inversely proportional to the size of its automorphism group as a $\mathbb Z_p[\sigma]$-module. 
In this section, we present some data to justify this assumption. For $p=3$ and $p=5$, we looked at the fields $\mathbb Q(\sqrt{-d})$ in which $p$ does not split and the $p$-part of the class group is cyclic of order $p^m$. For $p=3$, we considered $m=1$ and $m=2$.
For $p=5$, we considered only $m=1$ because the computation of $A_1$ takes much longer.
For each case we tabulated how often $A_1$ has one of the possibilities given in Theorem~\ref{main theorem}.
The results are given in the following tables.
We use the notation $p^3\times p\times p$, for example, to denote the group $(\mathbb Z/p^3\mathbb Z)\times (\mathbb Z/p\mathbb Z) \times (\mathbb Z/p\mathbb Z)$.

\subsection*{Calculations for the predicted value in the tables below}
Let us consider the case $m=2$ and $A_1\simeq 3^3\times 3\times 3$.
By Lemma~\ref{autsize}, the order of the automorphism group of $3^3\times 3\times 3$ is $2\cdot 3^4$.
There are two values of $j$ that give an ideal that yields this possibility for $A_1$.
Therefore, the predicted frequency of this group is 
\[
\frac{2/\abs{\Aut(3^3\times 3\times 3)}}{ \displaystyle\sum_{ \abs{A_0}=3^2} \frac{1}{\abs{\Aut(A_1)}} }= \frac{2/(2\cdot 3^4)}{1/(2\cdot 3^{2-1})}= \frac{2}{27}\approx 0.0741,
\]
where the denominator was evaluated in \eqref{used later} in Section \ref{compatibility}.
This gives one of the predicted values in Table~\ref{table 3}.  

Now consider $A_0$ cyclic of order 5 and $\abs{A_1}=5^5$.
There are two possibilities for $A_1$, namely, $25\times 5\times 5\times 5$ and $5\times 5\times 5\times 5\times 5$.
Both have automorphism groups of size $4\cdot 5^4$.
However, the first group occurs for 3 values of $j$ and the second group occurs only for  $j\equiv (-1)^m\equiv -1$ mod $5$.
So the predicted frequency of $25\times 5\times 5\times 5$ is
\[
\frac{3/\abs{\Aut(5^2\times 5\times 5\times 5)}}{ \displaystyle\sum_{ \abs{A_0}=5} \frac{1}{\abs{\Aut(A_1)}} }= \frac{3/(4\cdot 5^4)}{1/(4\cdot 5^{1-1})}= \frac{3}{625}= 0.0048,
\]
while the predicted frequency of $5\times 5\times 5\times 5\times 5$ is
\[
\frac{1/\abs{\Aut(5^2\times 5\times 5\times 5)}}{ \displaystyle\sum_{ \abs{A_0}=5} \frac{1}{\abs{\Aut(A_1)}} }= \frac{1/(4\cdot 5^4)}{1/(4\cdot 5^{1-1})}= \frac{1}{625}= 0.0016.
\]
These are two of the predicted values in Table~\ref{table 4}.

\subsection*{Tables}

No groups of higher order occurred in the range of the calculations.
This was expected because these groups had a very small chance of occurring without extensive calculations.
The group $25\times 25\times 25\times 5$  did not occur in the range listed in the table.
However, it occurs, for example, when $d=-513092$.

\begin{table}[h!]
\begin{tabular}{|c|c|c|c|c|c|c|c|}
\hline
 & Number of $d$ & 9 & $9\times 3$ & $3\times 3\times 3$ & $9\times 9$ & $27\times 9$ & $27\times 27$ \\
\hline
$3\nmid d$ & 18315 & 0.6669 & 0.1118 & 0.1104 & 0.0728 & 0.0267 & 0.0079 \\
$3\mid d$ & 12096 & 0.6685 & 0.1132 & 0.1122 & 0.0703 & 0.0227 & 0.0091 \\
Predicted  &    & 0.6667 & 0.1111 & 0.1111 & 0.0741 & 0.0247 & 0.0082 \\
\hline
\end{tabular}


\begin{tabular}{|c|c|c|c|c|c|}
\hline
$81\times 27$ & $81\times 81$ & $3^5\times 3^4$ & $3^5\times 3^5$ & $ 3^6\times 3^5$\\
\hline
 0.0023 & 0.0008 & 0.0003 & 0.0001 & 0.0001 \\
 0.0027 & 0.0010 & 0.0003 & 0.0000 & 0.0000 \\
 0.0027 & 0.0009 & 0.0003 & 0.0001 & 0.0000 \\
\hline
\end{tabular}

\medskip
 
\caption{$A_0$ is cyclic of order 3. Distribution of possible $A_1$ for fundamental discriminants of the form $-1-3j$ for $10^6 \le j \le 10^6+2\times 10^5$ (the line $3\nmid d$) and of the form $-3j$ for $10^6 \le j \le 10^6+2\times 10^5$ (the line $3\mid d$).}
\end{table}

\begin{table}[h!]
\begin{tabular}{|c|c|c|c|c|c|c|c|}
\hline
 & Number of $d$ & 27 & $27\times 3$ & $27\times 3\times 3$ & $9\times 9\times 9$ & $27\times 9\times 3$  \\
\hline
$3\nmid d$ & 6207 & 0.6742 & 0.2178 & 0.0723 & 0.0130 & 0.0098 \\
$3\mid d$ & 4181 & 0.6683 & 0.2148 & 0.0737 & 0.0148 & 0.0134 \\
Predicted  &    & 0.6667 & 0.2222 & 0.0741 & 0.0123 & 0.0123  \\
\hline
\end{tabular}


\begin{tabular}{|c|c|c|c|c|c|}
\hline
$27\times 27\times 3$ & $81\times 27\times 3$ & $81\times 81\times 3$ & $3^5\times 3^4\times 3$ & $3^5\times 3^5\times 3$ & $ 3^6\times 3^5\times 3$\\
\hline
 0.0079 & 0.0027 & 0.0013 & 0.0006 & 0.0002 & 0.0000\\
 0.0110  & 0.0024 & 0.0010 & 0.0002 & 0.0002 & 0.0002 \\
 0.0082 & 0.0027 & 0.0009 & 0.0003 & 0.0001 & 0.0000 \\
\hline
\end{tabular}

\medskip

\caption{$A_0$ is cyclic of order 9. Distribution of possible $A_1$ for fundamental discriminants of the form $-1-3j$ for $10^6 \le j \le 10^6+2\times 10^5$ (the line $3\nmid d$) and of the form $-3j$ for $10^6 \le j \le 10^6+2\times 10^5$ (the line $3\mid d$).}
\label{table 3}
\end{table}


\begin{table}[h!]
\begin{tabular}{|c|c|c|c|c|c|c|c|}
\hline
& Number of $d$ & $25$ & $25\times 5$ & $25\times 5\times 5$ & $25\times 5\times 5\times 5$  \\
\hline
$-2-5k$ & 588 &  0.8078 & 0.1582 & 0.0272  & 0.0051 \\
$-3-5k$ & 561 & 0.7843 & 0.1765 & 0.0196 & 0.0143\\
$-5k$ & 482 & 0.8050 & 0.1515 & 0.0353 & 0.0083\\
Predicted &  & 0.8000 & 0.1600 & 0.0320 & 0.0048\\
\hline
\end{tabular}
\end{table}

\begin{table}[h!]
\begin{tabular}{|c|c|c|c|c|}
\hline
$ 5\times 5\times 5\times 5\times 5$ & $25\times 25\times 5\times 5$ & $25\times 25\times 25\times 5$ & $25\times 25\times 25\times 25$ \\
\hline
0.0000 & 0.0017 & 0.0000 & 0.0000 \\
0.0018 & 0.0018 & 0.0000 & 0.0018\\
0.0000 & 0.0000 & 0.0000 & 0.0000\\
0.0016 & 0.0013 & 0.0003 & 0.0001\\
\hline
\end{tabular}

\medskip

\caption{$A_0$ is cyclic of order 5. Distribution of possible $A_1$ for fundamental discriminants of the form $-2-5k$, $-3-5k$, and $-5k$ for $10^6 \le k < 10^6+10^4$.}
\label{table 4}
\end{table}


\section{Comparing cyclotomic results with the anti-cyclotomic setting}
\label{comparison sec}

In \cite{KW23}, we studied the Iwasawa invariants in the anti-cyclotomic $\Z_p$-extension $K_{\ac}$ of an imaginary quadratic field $K$ where $p$ is non-split by studying the class groups of the base field $K$ and of the layers $K_1$ (and, when possible, $K_2$) in this anti-cyclotomic tower.
In this section we highlight some of the key differences.

In the anticyclotomic case, the map $A_0\to A_1$ is often non-injective.
When this happens, Theorem \ref{main theorem} does not apply.
In fact, a common occurrence for the anticyclotomic $\mathbb Z_3$-extensions was $A_0\simeq \mathbb Z/3\mathbb Z$ with $A_1\simeq (\mathbb Z/3\mathbb  Z)^2$, something  that cannot occur for cyclotomic $\mathbb Z_3$-extensions of imaginary quadratic fields where $3$ is non-split.
This phenomenon can happen only if the map $A_0\to A_1$ is non-injective or $3$ splits.

The Iwasawa invariants in the cyclotomic and the anti-cyclotomic $\Z_p$-extensions$K_{\cyc}$ and $K_{\ac}$ of the imaginary quadratic field $K$ behave differently and this is due to two main reasons:
First, the Iwasawa module $X(K_{\cyc})$ has no non-trivial finite $\Lambda$-submodules, whereas there are lots of examples where $X(K_{\ac})$ has finite $\Lambda$-submodules.
Second, the $p$-Hilbert class field of $K$ (denoted by $H(K)$) is always disjoint from $K_{\cyc}$, but this need not be true for $K_{\ac}$.

In \cite[Theorem~6.1]{KW23} we proved the following result: Let $K$ be an imaginary quadratic field and let $p$ be an odd prime that does not split in $K$. Suppose that $A_0$ is non-trivial.
If $H(K)$ is disjoint from $K_{\ac}$, then the $p$-part of the class group for the $n$-th layer of the anti-cyclotomic tower is \textit{non-cyclic} for all $n \ge 1$. This is in contrast to the cyclotomic case, where cyclic $A_1$ is quite common.

On the other hand, S.~Fujii \cite{Fuj13}  proved that when $H(K) \subset K_{\ac}$, the invariants $\lambda_{\ac}$ and $\mu_{\ac}$ are $0$.
In fact, it was proved that $A_n=0$ for sufficiently large $n$, so the size of $A_n$ decreases from $\abs{A_0}$ as $n$ increases, in contrast to the cyclotomic case, where $|A_n|$ strictly increases as $n$ increases.

\section{Function fields}\label{ff}

Suppose $K=K_0$ is a function field over a finite field and $K_{\infty}/K$ is a $\mathbb Z_p$-extension.
If $p$ is not the characteristic of $K$, then the only choice is obtained by a constant field extension.
However, if $p$ is the characteristic of $K$, there is a wide choice of such extensions; see \cite{GK}.
As proved in \cite{MW}, the map $A_n\to A_{n+1}$ is injective, which is one of the key ingredients needed in the characteristic 0 situation, especially to prove that $\lambda=1$, $\mu=0$ is equivalent to $e_1-e_0=1$. 

The Cohen--Lenstra--Martinet heuristics appear to be  more approachable in the function field case, and progress has been made when $p$ is not the characteristic; see, for example, \cite{A}.
The case where $p$ is the characteristic of the field is more difficult (see, \cite{CEZ}), but it would be useful to put the considerations of the present section into a  theoretical framework.

\subsection{Data}
We computed one family of examples.
Let $K=\mathbb F_3(X)$ and $K_1=\mathbb F_3(T)$, where $T^3-T+\frac{1}{X}=0$.
Then $K_1/K$ is the start of a (actually, many) $\mathbb Z_3$-extension
in which only the prime at $0$ ramifies and it is totally ramified.
Consider quadratic fields of the form $K(\sqrt{h(X)})$, where $h(X)$ is a square-free polynomial of odd degree such that $h(0)=0$.
These fields are analogous to imaginary quadratic fields in which $0$ is non-split, and $K_1(\sqrt{h(X)})$ is the first layer of the $\mathbb Z_3$-extension lifted to the quadratic field.
We let $h(X)$ run through monic polynomials in $\mathbb F_3[X]$ of degree 7, removing duplicate fields caused by reversing the polynomial (via $X^8h(1/X)$).
We then computed the class number of each $K(\sqrt{h(X)})$.
If 3 divided this class number, we computed the class number of $K_1(\sqrt{h(X)})$.
From Theorem~\ref{main theorem}, this class number determines the structure of $A_1$, except when $\abs{A_1}=3^{3m}$, where there are two choices.

We computed for a total of 200 polynomials $h(X)$.
Of these, 73 had $3\mid h_0$, which is lower than the Cohen--Lenstra prediction of $88$, which is probably another example of the slow convergence of these heuristics.
Of these 73 fields, 57 had $e_0=1$ and 16 had $e_0=2$.
Since $e_0=2$ does not distinguish between cyclic and non-cyclic, we restrict to $e_0=1$.
Of these 57 fields, 44 have $e_1=2$, which is equivalent to $\lambda=1$. 
This is somewhat higher than the predicted $(2/3)57=38$.
Again, this could be due to the slow convergence of the heuristics, or there could be another phenomenon present.

The data for these 57 fields is given in Table~\ref{table 5}.

\begin{table}[h!]
\begin{tabular}{|c|c|c|c|c|c|c|c|}
\hline
& $9$ & $9\times 3$ or $3\times 3\times 3$ & $9\times 9$ & $27\times 9$ & $27\times 27$  \\
\hline
Computed  & 0.77 &  0.16 & 0.04 & 0.00  & 0.04 \\
Predicted   & 0.67 & 0.22 & 0.07 & 0.02 & 0.01\\
\hline
\end{tabular}
\medskip
\caption{Distribution of 3-parts of $A_1$ for the 57 polynomials $h(X)$ of degree 7 where $A_0$ is cyclic of order 3. }
\label{table 5}
\end{table}
\medskip

\bibliographystyle{amsalpha}
\bibliography{references}

\end{document}